\documentclass{article}
\usepackage{amsmath}
\usepackage{amssymb}
\usepackage{titlesec}
\usepackage{graphicx}
\usepackage{fancyhdr}
\usepackage{latexsym}
\usepackage{mathrsfs}
\usepackage{booktabs}
\usepackage{mathrsfs}%����
\usepackage{mathtools}
\usepackage{amsthm}
\usepackage{wasysym}
\usepackage{cite}
\usepackage{color}
\usepackage{amsfonts}
\usepackage{amsmath,amssymb}
\usepackage{graphics}
\usepackage{multimedia}
\usepackage{graphics}
\usepackage{cite}
\usepackage{latexsym}
\usepackage{amsmath}
\usepackage{amssymb}
\usepackage[dvips]{epsfig}
\usepackage{amscd}
\numberwithin{equation}{section}

\newtheorem{theorem}{Theorem}[section]
\newtheorem{lemma}{Lemma}[section]
\newtheorem{remark}{Remark}[section]
\newtheorem{proposition}{Proposition}[section]

\newcommand{\dv}{\text{div}}

\title{Global large strong solutions to the compressible Navier-Stokes equations with density-dependent viscosities, case I: isentropic flows}
\author{Xiangdi H{\small UANG}$^{c}$, Jiaxu L{\small I}$^{b}$, Rong Z{\small HANG}$^{a,d}$ \thanks{Email addresses: xdhuang@amss.ac.cn (X. D. Huang), Jiaxvlee@gmail.com (J. X. Li), rzhang0921@gmail.com (R. Zhang). }  \\ 
{\normalsize a. School of Mathematics and Computer Sciences,}\\
{\normalsize Nanchang University, Nanchang 330031, P. R. China;}\\
{\normalsize b.  The Institute of Mathematical Sciences,}\\
{\normalsize  The Chinese University of Hong Kong,
Shatin, N.T.;}\\
{\normalsize c. Institute of Mathematics, Academy of Mathematics and Systems Sciences,}\\
{\normalsize Chinese Academy of Sciences, Beijing 100190, China;}\\
{\normalsize d. Institute of Mathematics and Interdisciplinary Sciences,}\\
{\normalsize Nanchang University, Nanchang 330031, P. R. China.}
}
\date{}

\begin{document}

\maketitle

\begin{abstract}
In this paper, we consider the Cauchy problem for the three-dimensional barotropic compressible Navier-Stokes equations with density-dependent viscosities.  
By considering the system as an elliptic-dominated structure and defining suitable energy functionals, after the elaborate index analysis, 
we establish the global existence of strong solutions as long as the initial data is large enough. This is a big contrast to the classical results where the initial data is a small perturbation of some resting states.

\end{abstract}

\section{Introduction}
We consider the three-dimensional isentropic compressible Navier-Stokes equations in a domain $\Omega\subset\mathbb{R}^3$:
\begin{equation}\label{cns}
\begin{cases}
    \rho_t+\mathrm{div}(\rho u)=0,\\
    (\rho u)_t+\mathrm{div}(\rho u\otimes u)+\nabla P=\mathrm{div}\mathbb{T},
\end{cases}
\end{equation}
where $t\geq0, x=(x_1,x_2,x_3)\in\Omega$ are time and space variables, respectively. $\rho=\rho(x,t)$ and $u=(u_1(x,t),u_2(x,t),u_3(x,t))$ represent, respectively, the density and the velocity. The pressure $P$ is given by
\begin{equation}
P=a\rho^\gamma,\ (a>0, \gamma>1),
\end{equation}
where $a$ is an entropy constant and $\gamma$ is the adiabatic exponent. $\mathbb{T}$ denotes the viscous stress tensor with the form
\begin{equation}
\mathbb{T}=2 \mu(\rho) \mathcal{D}u +\lambda(\rho)\mathrm{div}u\mathbb{I}_3,
\end{equation}
where $\mathcal{D}u =\frac{1}{2}\left(\nabla u+(\nabla u)^\top \right)$ is the deformation tensor and $\mathbb{I}_3$ is the three-dimensional identity matrix. The viscosities $\mu(\rho)$ and $\lambda(\rho)$ satisfy the following hypothesis:
\begin{equation}\label{vis-c}
\mu(\rho)=\mu\rho^\alpha,\ \lambda(\rho)=\lambda\rho^\alpha,
\end{equation}
with constants $\alpha\geq 0$ and
\begin{equation}\label{pc}
\mu>0,\ 2\mu+3\lambda\geq 0.
\end{equation}

Let $\Omega=\mathbb{R}^3$. We look for a global strong solution $(\rho,u)$ to the Cauchy problem \eqref{cns}-\eqref{pc} with the following initial data and far-field condition:
\begin{equation}
(\rho,u)|_{t=0}=(\rho_0,u_0),\quad x\in\mathbb{R}^3,
\end{equation}
\begin{equation}\label{7}
(\rho-\Bar{\rho},u)(t,x)\ \rightarrow 0,\quad \text{as}\ |x|\rightarrow\infty,\ t\geq0.
\end{equation}

%{\color{red}review and motivation....}\\

When both the shear and bulk viscosities are positive constants $(\alpha=0)$, there is a lot of literature on the well-posedness of solutions to the problem \eqref{cns}-\eqref{pc} in multi-dimensional space, see \cite{Matsumura1980,Lions1993,Lions1998-2,Feireisl2001,HuangLiXin2012} and the references therein. 
In particular, 
%the local existence and uniqueness of classical solutions are proved in \cite{Nash1962,71Itaya} for the absence of vacuum and in \cite{03jde-choe-kim,04jmpa-cck,06mm-ck,Huang2020} for the presence of vacuum. 
the first global result is obtained by Matsumura-Nishida \cite{Matsumura1980} where they prove the global existence of classical solutions for initial data close to a nonvacuum equilibrium in some Sobolev space $H^s$. 
For large initial data, the global existence of weak solutions was first obtained by Lions \cite{Lions1993,Lions1998-2} provided $\gamma\geq\frac{9}{5}$, which is improved later by Feireisl-Novotny-Petzeltov\'a \cite{Feireisl2001} for $\gamma>\frac{3}{2}$. 
Recently, for the case that the initial density is allowed to vanish, Huang-Li-Xin \cite{HuangLiXin2012} obtains the global existence of classical solutions to the Cauchy problem for the barotropic compressible Navier–Stokes equations in three spatial dimensions with smooth initial data provided that the initial energy is suitably small.

In the theory of gas dynamics, the non-isentropic compressible Navier-Stokes can be derived from the Boltzmann equations through the Chapman-Enskog expansion, cf. Chapman-Cowling \cite{chapman1990mathematical} and Li-Qin \cite{li2012physics}.
For the well-known Sutherland's model, the viscosity coefficients are not constants but functions of the absolute temperature $\theta$ such as:
$$\mu(\theta)= \frac{c_1\theta^\frac{3}{2}}{\theta+s_0}, \quad \lambda(\theta)=\frac{c_2\theta^\frac{3}{2}}{\theta+s_0}, \quad s_0>0\ \mathrm{is}\ \mathrm{a}\ \mathrm{constant}.$$
%In particular, for Maxwellian molecules, a = 4; for rigid elastic spherical molecules, a = ∞; while for ionized gas, a = 1 (see §10 of [5]). 
In the meanwhile, according to Liu-Xin-Yang \cite{liu1997vacuum}, for isentropic and polytropic fluids, such dependence is inherited through the laws of Boyle and Gay-Lussac:
$$P=R\rho \theta=a\rho^\gamma,$$
for constant $R>0$, and one finds that the viscosity coefficients are functions of the density.

When the shear and bulk viscosities are formulated as powers of density, there are also many results on the well-posedness problems of \eqref{cns}. 
For the shear viscosity is constant and the second viscosity is the power of density ($\lambda=\rho^\beta$), Kazhikhov-Vaigant \cite{kazhikhov1995existence} firstly established the global existence of strong solutions in 2-dimensional periodic case, which is generalized to whole space \cite{huang2022global} and bounded domains \cite{fan2022global} recently.

If both viscosities are powers of the density, recent studies by Li-Xin \cite{15li-xin} as well as Vasseur-Yu \cite{Vasseur2016,bresch2021global} have independently investigated the global existence of weak solutions for compressible Navier-Stokes systems with viscosities satisfying the B-D condition \cite{bresch2004some}. 
Concerning the strong solutions with vacuum at infinity, for the case $\alpha> 1$, Li-Pan-Zhu \cite{19arma-li-pan-zhu} has demonstrated the existence of local classical solutions with vacuum, see \cite{17jmfm-li-pan-zhu} for $\alpha=1$. 
In the case of $\alpha\in(0,1)$, Xin-Zhu \cite{zhuxin2021} establishes a local regular solution with a far-field vacuum. While for the global existence of strong solutions, for $\alpha>1$, Xin-Zhu \cite{XIN2021108072} has also shown the global existence of regular solutions with vacuum at infinity if $\alpha>2\gamma-1$ for a class of smooth initial data that are of small density but possibly large velocities.

A natural question arises of whether large global solutions exist or not. 
Recently, Yu \cite{yu2023global} firstly obtains the global existence of large strong solutions if the initial data $(\rho_0,u_0)$ satisfy
\begin{equation*}
    \frac{3}{4} \bar \rho \le \rho_0\le \frac{5}{4} \bar \rho,\quad \rho_0-\bar \rho \in H^1 \cap D^{1,4}, \quad  u_0 \in H^2,
\end{equation*}
and $\alpha, \gamma$ satisfy 
\begin{equation}\label{yu1}
    \frac43<\gamma\le \alpha\le \frac53, \alpha+4\gamma>7, \alpha+\gamma\le 3,
\end{equation}
or 
\begin{equation}\label{yu2}
    \frac32<\gamma\le \alpha\le 2, \alpha+\frac23 \gamma>3,
\end{equation}
and the initial density is large enough. 
However, it should be noticed that the constraints on the initial data and $\alpha,\gamma $ in \cite{yu2023global} are too demanding and restrictive, which excludes many physical cases. 

In this paper, we aim to obtain the global existence of strong solutions to \eqref{cns} under more general initial data and less restrictions on $\alpha,\gamma$ in order to be applicable for more physical cases.

Indeed, we find another total different iteration scheme and we can obtain the existence of large strong solutions to the Navier-Stokes equations \eqref{cns} with density-dependent viscosities if the initial density is large enough under the assumption \eqref{ccc} (see below). This discovery also coincides with the phenomena showed in constant viscosities case that the global existence of strong solutions can be obtained if the viscosity is large enough, see Deng-Zhang-Zhao \cite{deng2012global}.

%Luo-Zhou [30] has further established the local existence of regular solutions when the shear and bulk viscosity coefficients are of different density powers. 

Before stating the main results, we explain the notation and conventions used throughout this paper. We denote
\begin{equation*}
    \int f dx =\int_{\mathbb{R}^3} f dx.
\end{equation*}
For $1\le r\le \infty$ and integer $k\ge 0$, we denote the standard homogeneous and inhomogeneous Sobolev spaces as follows:
\begin{equation*}
\begin{gathered}
L^r=L^r(\mathbb{R}^3), \quad D^{k,r}=\{u\in L^1_{loc}(\mathbb{R}^3)\ |\  \|\nabla^k u\|_{L^r}<\infty\},\\
W^{k,r}=L^r \cap D^{k,r}, \quad H^k=W^{k,2}.
\end{gathered}
\end{equation*}

Now we are ready to present our main theorem.
\begin{theorem}\label{th1}
	Let $\Omega=\mathbb{R}^3$ and $\bar \rho>1$ be a constant.
    Suppose that the initial data $(\rho_0,u_0)$ satisfies 
	\begin{equation}\label{ia}
		\frac{3}{4} \bar \rho \le \rho_0\le \frac{5}{4} \bar \rho,\quad \rho_0-\bar \rho \in L^{\gamma} \cap %D^{1,2} \cap 
        D^{1,q},\ 3<q<6,\quad  u_0 \in H^2.
	\end{equation}
    Let $\alpha$ and $\gamma$ satisfy
	\begin{equation}\label{restri}
		\alpha>1,\ \gamma>1, %2\alpha>1+\gamma,
	\end{equation}
    and 
    \begin{equation}\label{ccc}
        \alpha > \max\left\{\frac{\gamma+1}{2} , \frac{5q}{2(4q-3)}(\gamma-1) \right\}.
    \end{equation}
    % \begin{cases}
    %         \alpha>\frac{1+\gamma}{2}, &1< \gamma\le 4,\\
    %         \alpha>\frac{5}{6}(\gamma-1), \quad &\gamma > 4.
    %     \end{cases}
    % \begin{equation}\label{ccc}
    %     \alpha\ge \frac{5q}{2(4q-3)}(\gamma-1).
    % \end{equation}
	 
    %and the compatibility condition
    % \begin{equation}
    %     \rho_0^{1/2} g = \rho_0^{-1/2}\left(-\rho_0 u_0\cdot \nabla u_0 - \nabla P(\rho_0) +2\mu \dv( \rho_0^{\alpha}\mathcal{D} u_0)+\lambda\nabla( \rho_0^{\alpha} \dv u_0)\right),
    % \end{equation}
    % for some $g \in L^2$. 
	Then there exists a positive constant $\Lambda_0 =\Lambda_0 (a,\mu,\lambda, \alpha,\gamma,\|\rho_0-\bar\rho\|_{L^{\gamma}}, \|\nabla \rho_0\|_{L^q}, \|u_0\|_{H^2})$ 
    such that if
	\begin{equation}
		\bar\rho \ge \Lambda_0,
	\end{equation}
	then the problem  (\ref{cns})--(\ref{7}) admits a unique global strong solution $(\rho,u)$ in $\mathbb{R}^3\times(0,\infty)$ satisfying that 
	\begin{equation}
		\frac{2}{3} \bar \rho \le \rho \le \frac{4}{3} \bar \rho,
	\end{equation}
	and 
	\begin{equation}
		\left\{ \begin{array}{l}
    \rho\in C([0,\infty);L^\gamma \cap D^{1,q}), \ \rho_t\in C([0,\infty);L^q),\\
    \nabla u \in C([0,\infty);H^1)\cap L^2((0,\infty);W^{1,q}),\\
    u_t\in L^\infty((0,\infty);L^2)\cap  
     L^2((0,\infty);H^1).
    \end{array} \right.
	\end{equation}
	
\end{theorem}

\begin{remark}
    Recall the result by Xin-Zhu \cite{XIN2021108072} where they obtain the global regular solutions if $\alpha>2\gamma-1$, which becomes a classical one if $1<\min\{\alpha,\gamma\}<3$. In this paper, we can obtain a large solution only under the assumption that $\alpha,\gamma$ satisfy \eqref{restri}-\eqref{ccc}. 
\end{remark}

\begin{remark}
    Compared to the result by Yu \cite{yu2023global}, where they required $\alpha, \gamma$ satisfying
    \eqref{yu1} and \eqref{yu2}. By using a different iteration scheme, we can extend the range of $\alpha,\gamma$ to more general cases. Moreover, the assumption on initial density $\|\nabla \rho_0\|_{L^2}$ is removed, which plays a key role in Yu's analysis. 
\end{remark}

\begin{remark}
    For the constant viscosities case \cite{deng2012global}, the global strong solution is obtained if the viscosity coefficient is large enough. For the density-dependent model, the Theorem \ref{th1} also implies that the global large solution exists if the viscosities are large enough.
\end{remark}

% \begin{remark}
%     In fact, the constraint $\alpha>\frac{5}{6}(\gamma-1)$ in \eqref{ccc} can be removed if we consider the initial data with lower regularity.
% \end{remark}

% \begin{remark}
%     $N=2$???
% \end{remark}

We now provide our analysis and commentary on the key aspects of this paper. For the density-dependent viscosities model, it's difficult to study the well-posedness due to the strong coupling between $\rho$ and $u$. Different from the approach \cite{17jmfm-li-pan-zhu,zhuxin2021,19arma-li-pan-zhu} where they consider the system from the 'quasi-symmetric
hyperbolic'-'degenerate elliptic' structure, we consider the system as elliptic-dominated equations. More precisely, rewrite the momentum equation as 
\begin{equation}\label{ellip}
    \mu \Delta u +(\mu+\lambda) \nabla \dv u  =\rho^{1-\alpha}\dot u
    +a \gamma\rho^{\gamma-\alpha-1} \nabla \rho
    -\alpha\rho^{-1}\nabla\rho\cdot(2\mu \mathcal{D}u +\lambda \dv u).
\end{equation}
where 
\begin{equation}
    \dot u\triangleq u_t+u\cdot\nabla u
\end{equation}
denote the material derivatives.
Therefore, if $\alpha>\max\{1,\gamma\}$ and $\rho>>1$, each term of the right-hand side of Lame system \eqref{ellip} can be viewed as a small perturbation term, this is the main observation of Yu \cite{yu2023global}. However, in this paper, we do not need $\alpha>\gamma$, thus we should design a new iteration scheme to deal with possible large oscillation of the pressure term in \eqref{ellip}. Moreover,
to obtain the global existence of strong solutions, the key is to derive the uniform in time estimates, especially, the upper bound of the density according to the blow-up criteria \cite{huang2011serrin}. For this purpose, we need uniform estimates of $\nabla \rho$.

The rest of the paper is organized as follows: In Section \ref{s2}, we collect some elementary facts and inequalities that will be needed in later analysis. Section \ref{s3} is devoted to the a priori estimates that are needed to obtain the global existence of strong solutions. Then the main result, Theorem \ref{th1} is proved in Section \ref{s4}.

\section{Preliminaries}\label{s2}

In this section, we will recall some known facts and elementary inequalities that will be used frequently later.

We start with the local existence and uniqueness of strong solutions.
Similar to Zhang and Zhao \cite{zhang2014existence}, we can construct a unique local strong solution as follows:

\begin{lemma}[Local Existence]\label{local}
    Assume that the initial data $(\rho_0, u_0)$ satisfies \eqref{ia}. 
    % the regularity condition
  %   \begin{equation}\label{ia}
		% \frac{1}{2} \bar \rho \le \rho_0\le \frac{3}{2} \bar \rho,\quad \rho_0-\rho \in %L^{\gamma} \cap D^{1,2} \cap 
		% W^{1,q},\ 3<q<6,\quad  u_0 \in  H^2,
  %   \end{equation}
  %   and the compatibility condition
  %   \begin{equation}\label{cc}
		% \rho_0^{1/2} g = \rho_0^{-1/2}\left(-\rho_0 u_0\cdot \nabla u_0 - \nabla P(\rho_0) +2\mu \dv( \rho_0^{\alpha}\mathcal{D} u_0)+\lambda\nabla( \rho_0^{\alpha} \dv u_0)\right),
  %   \end{equation}
  %   for some $g\in L^2$. 
    Then there exists a small time $T$ and a unique strong solution $(\rho, u)$ to the Cauchy problem (\ref{cns})--(\ref{vis-c}) such that
    \begin{equation}\label{l-r}
    \left\{ \begin{array}{l}
    \rho\in C([0,T];D^{1,q}),\
    \nabla u \in C([0,T];H^1)\cap L^2([0,T];W^{1,q}),\\
    \rho_t\in C([0,T];L^q),\
    \sqrt{\rho}u_t\in L^\infty([0,T];L^2),\ 
    u_t\in L^2([0,T];H^1).
    \end{array} \right.
\end{equation}
% Furthermore, if $T^\ast$is the maximal existence time of the local strong solution $(\rho, u)$, then either
% $T^\ast=\infty$ or
% \begin{equation}\label{blow-up}
%     \sup_{0\leq t\leq T^\ast}\left(\|\nabla\rho\|_{L^q}+\|\nabla u\|_{L^2}\right)=\infty.
% \end{equation}
\end{lemma}

Next, the following well-known Gagliardo-Nirenberg inequality  will be used frequently later (see \cite{ladyzhenskaia1968linear}).
\begin{lemma}[Gagliardo-Nirenberg Inequality]\label{G-N}
Let 1 $\leq q  +\infty$ be a positive extended real quantity. Let $j$ and $m$  be non-negative integers such that $ j < m$. Furthermore, let $1 \leq r \leq \infty$ be a positive extended real quantity,  $p \geq 1$  be real and  $\theta \in [0,1]$  such that the relations
\begin{equation}
    \dfrac{1}{p} = \dfrac{j}{n} + \theta \left( \dfrac{1}{r} - \dfrac{m}{n} \right) + \dfrac{1-\theta}{q}, \qquad \dfrac jm \leq \theta \leq 1,
\end{equation}
hold. Then, 
\begin{equation}
    \|\nabla^j u\|_{L^p(\mathbb{R}^n)} \leq C\|\nabla^m u\|_{L^r(\mathbb{R}^n)}^\theta\|u\|_{L^q(\mathbb{R}^n)}^{1-\theta},
\end{equation}
where $u \in L^q(\mathbb{R}^n)$  such that  $\nabla^m u \in L^r(\mathbb{R}^n)$. Moreover, if $q>1$ and $r>3$,
\begin{equation}
\|u\|_{C(\overline{\mathbb{R}^n})}\leq C\|u\|^{q(r-3)/(3r+q(r-3))}_{L^q(\mathbb{R}^n)}\|\nabla u\|^{3r/(3r+q(r-3))}_{L^r(\mathbb{R}^n)},
\end{equation}
where $u \in L^q(\mathbb{R}^n)$  such that  $\nabla u \in L^r(\mathbb{R}^n)$.
The constant $C > 0$  depends on the parameters $j,\,m,\,n,\,q,\,r,\,$ and $\theta$.

\end{lemma}

\section{A priori estimates}\label{s3}

For any fixed time $T>0,$ let $(\rho,u)$ be a strong solution to \eqref{cns}-\eqref{7} on $\mathbb{R}^3 \times (0,T]$ with initial data $(\rho_0,u_0)$ satisfying \eqref{ia}.

Define
\begin{gather}\label{As1}
    \mathcal{E}_{\rho,1}(T) \triangleq \sup_{t\in[0,T] }\|\rho-\bar\rho \|_{L^\infty},\\
	%\mathcal{E}_{\rho,2}(T) \triangleq \sup_{t\in[0,T] }\|\nabla \rho \|_{L^2}^2 + \bar \rho^{\gamma-\alpha} \int_0^{T}  \|\nabla \rho \|_{L^2}^2dt,\\
	\mathcal{E}_{\rho,2}(T) \triangleq \sup_{t\in[0,T] }\|\nabla \rho \|_{L^q}^2
	+ \bar \rho^{\gamma-\alpha} \int_0^{T}  \|\nabla \rho \|_{L^q}^2dt,\\
	\mathcal{E}_{u,1}(T) \triangleq \mu \frac{\bar\rho^{\alpha} }{2^{\alpha+1}} \sup_{t\in[0,T] }\|\nabla u \|_{L^2}^2
	+ \frac 12  \int_0^{T}  \|\sqrt{\rho} u_t \|_{L^2}^2dt,\\
	\mathcal{E}_{u,2}(T) \triangleq \sup_{t\in[0,T] }\|\sqrt{\rho} u_t \|_{L^2}^2
	+ \mu \frac{\bar\rho^{\alpha} }{2^{\alpha+1}}  \int_0^{T}  \|\nabla u_t \|_{L^2}^2dt.
\end{gather}

We have the following key proposition.

\begin{proposition}\label{pr}
	Under the conditions of Theorem \ref{th1}, for 
    \begin{equation}\label{b}
        \beta=\max\{3-\gamma,0\}
    \end{equation}
    % with
    % \begin{equation}\label{delta}
    %     \delta \in (0,\min\{\gamma-1,\alpha-1, 2\alpha-\gamma-1\}),
    % \end{equation}
	there exist  positive constants $N_2=N_2(\mu, \lambda, \|\nabla u_0\|_{L^2})$, $N_i=N_i(a,\mu,\lambda, \alpha,\gamma,\|\rho_0-\bar\rho\|_{L^{\gamma}},\\ \|\nabla \rho_0\|_{L^q}, \|\nabla u_0\|_{H^1}), i=1,3,$ and  $\Lambda_0 =\Lambda_0 (a,\mu,\lambda, \alpha,\gamma,\|\rho_0-\bar\rho\|_{L^{\gamma}}, \|\nabla \rho_0\|_{L^q}, \|u_0\|_{H^2} )$ %depending on $\mu,\ \lambda,\ a,\  \gamma,\ \alpha$, and initial data 
    such that if $(\rho,u)$ is a smooth solution to the problem (\ref{cns})--(\ref{7}) on $\mathbb{R}^3 \times (0,T]$ satisfying
	\begin{equation}\label{a1}
        \begin{gathered}
        \mathcal{E}_{\rho,1}(T) \le \frac{\bar\rho}{2}, \quad\mathcal{E}_{\rho,2}(T)  \le 2N_1
         \bar\rho^{\beta}, \\
		\mathcal{E}_{u,1}(T) \le 2^{\alpha+2}N_2 \bar\rho^{\alpha}, 
        \quad\mathcal{E}_{u,2}(T) \le 3 N_3 \bar\rho^{2\alpha-1},   
        \end{gathered}
	\end{equation}
	the following estimates hold:
	\begin{equation}\label{a2}
        \begin{gathered}
        \mathcal{E}_{\rho,1}(T) \le \frac{\bar\rho}{4}, \quad\mathcal{E}_{\rho,2}(T) \le N_1 \bar\rho^{\beta}, \\
		\mathcal{E}_{u,1}(T) \le 2^{\alpha+1}N_2 \bar\rho^{\alpha}, 
        \quad\mathcal{E}_{u,2}(T) \le 2 N_3 \bar\rho^{2\alpha-1},
        \end{gathered}
	\end{equation}
	provided 
	\begin{equation}
		\bar\rho\ge \Lambda_0,
	\end{equation}  
    with $\Lambda_0$ defined in \eqref{lambda}.
\end{proposition}

First, we have the following basic energy estimate.
\begin{lemma}
	Under the assumption that 
    $$\mathcal{E}_{\rho,1}(T) \le \frac{\bar\rho}{2},$$ 
    then it holds that
	\begin{equation}\label{ee}
		\sup_{0\le t\le T} \left(\bar \rho \| u\|_{L^2}^2  + \|\rho- 
            \bar \rho \|_{L^\gamma}^\gamma \right) + \bar\rho^{\alpha} \int_{0}^{T} \|\nabla u\|_{L^2}^2 dt \le C\bar\rho,
	\end{equation}
    with $C=C(a,\lambda,\mu, \alpha,\gamma, \|u_0\|_{L^2}, \|\rho_0-\bar \rho\|_{L^{\gamma}})$.
\end{lemma}

\begin{proof}
	First, it follows from \eqref{a1} that 
    \begin{equation}\label{upbd}
        \frac{\bar \rho }{2} \le \rho \le \frac{3\bar \rho}{2}.
    \end{equation}
    Then, multiplying \eqref{cns}$_2$ by $u$ and integrating the resultant equation, we obtain after using \eqref{cns}$_1$ and integration by parts that 
    \begin{equation}
        \begin{aligned}
            \frac{d}{dt} \left(\frac{1}{2}\|\sqrt{\rho} u\|_{L^2}^2+ \|G(\rho)\|_{L^1}\right) + \frac{\mu}{2^\alpha} \bar \rho^{\alpha}\|\nabla u\|_{L^2}^2 + \frac{\mu+\lambda}{2^\alpha} \bar \rho^{\alpha}\|\dv u\|_{L^2}^2\\
            \le \frac{d}{dt} \left(\frac{1}{2}\|\sqrt{\rho} u\|_{L^2}^2+ \|G(\rho)\|_{L^1}\right) + \int \left(
            2\mu \rho^{\alpha}|\mathcal{D} u|^2 + \lambda  \rho^{\alpha}(\dv u)^2\right) dx=0,
        \end{aligned}
    \end{equation}
    where 
    \begin{equation*}
        G(\rho)=\rho \int_{\bar \rho }^{\rho} \frac{P(s)-P(\bar\rho)}{s^2}ds.
    \end{equation*}
    Integrating the above inequality over $(0,T]$ leads to 
    \begin{equation}
        \begin{aligned}
            \sup_{0\le t\le T} \left(\bar \rho \| u\|_{L^2}^2  + \|\rho- 
            \bar \rho \|_{L^\gamma}^\gamma \right) + \bar\rho^{\alpha} \int_{0}^{T} \|\nabla u\|_{L^2}^2 dt \le C(\bar\rho \| u_0\|_{L^2}^2 + \|\rho_0- 
            \bar \rho \|_{L^\gamma}^\gamma),
        \end{aligned}
    \end{equation}
    where we have used 
    \begin{equation*}
        G(\rho) \sim (\rho-\bar\rho)^\gamma.
    \end{equation*}
    for $\bar \rho \ge 1$, \eqref{ee} follows directly. 
\end{proof}

Next, rewrite \eqref{cns}$_2$ as 
\begin{equation}\label{elliptic}
    \mu \Delta u +(\mu+\lambda) \nabla \dv u =  \nabla \mathcal{P}+H,
\end{equation}
with
\begin{equation*}
    \nabla \mathcal{P}=
        \begin{cases}
            \frac{a\gamma}{\gamma-\alpha} \nabla \rho^{\gamma-\alpha}, \quad \gamma \neq \alpha,\\
            a\gamma \nabla \ln \rho, \quad \gamma=\alpha,
        \end{cases}
\end{equation*}
and
$$H \triangleq \rho^{1-\alpha}\dot u
+\alpha\rho^{-1}\nabla\rho\cdot(2\mu \mathcal{D}u +\lambda \dv u).$$
Define the efficient viscous flux $F$ and vorticity $w$ respect to \eqref{elliptic} as
\begin{equation}\label{flux}
    F=(2\mu+\lambda) \dv u +  \mathcal{P}(\rho)-\mathcal{P}(\bar \rho),\quad w=\nabla \times u,
\end{equation}
then \eqref{elliptic} implies
\begin{gather}
    \label{f1}-\Delta F=\dv H,\\
    -\mu \Delta w =\nabla \times H.
\end{gather}

According to the standard $L^2$-estimate of the elliptic system, one has after using \eqref{upbd}, \eqref{ee}, and Sobolev's inequality in Lemma \ref{G-N},
\begin{equation}\label{e1}
    \begin{aligned}
        \|\nabla u\|_{L^6} \le& C (\|F\|_{L^6} + \|w\|_{L^6}+ \|\mathcal{P}(\rho)-\mathcal{P}(\bar \rho)\|_{L^6}) \\
        \le & C( \|\nabla F\|_{L^2} + \|\nabla w\|_{L^2}+ \bar \rho^{-\alpha}\|P(\rho)-{P}(\bar \rho)\|_{L^6} )\\
        \le & C \left(\bar\rho^{-\alpha+\frac12}\|\sqrt{\rho}u_t\|_{L^2}+ \bar\rho^{-\alpha+1} \|\nabla u\|_{L^2}^{\frac{3}{2}} \|\nabla u\|_{L^6}^{\frac{1}{2}} + \bar\rho^{-1} \|\nabla \rho \|_{L^q} \|\nabla u\|_{L^{2}}^{\frac{q-3}{q}} \|\nabla u\|_{L^6}^{\frac{3}{q}} \right) \\
        &+ C \bar \rho^{-\alpha}\|P(\rho)-{P}(\bar \rho)\|_{L^1}^{\frac{3(q-2)}{2(4q-3)}} \|\nabla P(\rho)\|_{L^q}^{\frac{5q}{2(4q-3)}}  \\
        \le & \frac{1}{2} \|\nabla u\|_{L^6} + C \left( \bar\rho^{-\alpha+\frac12}\|\sqrt{\rho}u_t\|_{L^2}+ \bar\rho^{-2\alpha+2} \|\nabla u\|_{L^2}^{3}  + \bar\rho^{-\frac{q}{q-3}} \|\nabla \rho \|_{L^q}^{\frac{q}{q-3}} \|\nabla u\|_{L^{2}} \right)\\
        &+ C \bar \rho^{-\alpha+\frac{3(q-2)}{2(4q-3)} + (\gamma-1) \frac{5q}{2(4q-3)} } \|\nabla \rho\|_{L^q}^{\frac{5q}{2(4q-3)}},
    \end{aligned}
\end{equation}
where we have used  
\begin{equation}\label{e2}
    \begin{aligned}
        &\|\nabla F\|_{L^2} + \|\nabla w\|_{L^2} \le C \|H\|_{L^2}\\
        \le& C \bar\rho^{-\alpha}\|-\rho(u_t+u\cdot \nabla u)+2\mu \nabla \rho^\alpha \cdot \mathcal{D}u +\lambda \nabla \rho^{\alpha} \dv u\|_{L^2}\\
        %\le& C \bar\rho^{-\alpha+\frac12}\|\sqrt{\rho}u_t\|_{L^2}+ \bar\rho^{-\alpha+1}  \|u\cdot \nabla u\|_{L^2} + \bar\rho^{-1} \|\nabla \rho \nabla u\|_{L^2}\\
        \le & C \left(\bar\rho^{-\alpha+\frac12}\|\sqrt{\rho}u_t\|_{L^2}+ \bar\rho^{-\alpha+1} \|\nabla u\|_{L^2}^{\frac{3}{2}} \|\nabla u\|_{L^6}^{\frac{1}{2}} + \bar\rho^{-1} \|\nabla \rho \|_{L^q} \|\nabla u\|_{L^{\frac{2q}{q-2}}} \right)\\
        \le & C \left( \bar\rho^{-\alpha+\frac12}\|\sqrt{\rho}u_t\|_{L^2}+ \bar\rho^{-\alpha+1} \|\nabla u\|_{L^2}^{\frac{3}{2}} \|\nabla u\|_{L^6}^{\frac{1}{2}} + \bar\rho^{-1} \|\nabla \rho \|_{L^q} \|\nabla u\|_{L^{2}}^{\frac{q-3}{q}} \|\nabla u\|_{L^6}^{\frac{3}{q}} \right).
    \end{aligned}
\end{equation}

Note that by Sobolev's inequality,
    \begin{equation}\label{rho1}
        \begin{aligned}
        \|\nabla u\|_{L^\infty} \le & C \|\nabla u\|_{L^2}^{\theta} \|\nabla^2 u\|_{L^q}^{1-\theta}\\
        %\le & C \|\nabla u\|_{L^2}^{\theta} \left( \bar\rho^{-\alpha}\|\rho u_t\|_{L^q} + \bar\rho^{\frac{-\alpha+1}{\theta}}  \| u\|_{L^q}^{\frac{1}{\theta}} \|\nabla u\|_{L^2} + \bar\rho^{-\frac{1}{\theta}} \|\nabla \rho\|_{L^q}^{\frac{1}{\theta}} \|\nabla u\|_{L^2}  + \bar \rho^{\gamma-\alpha-1}\|\nabla \rho \|_{L^q}\right)^{1-\theta}
        \end{aligned}
    \end{equation}
    where 
    \begin{equation}\label{theta}
        \theta=\frac{2(q-3)}{5q-6}.
    \end{equation}
    
    Similarly, the standard $L^p$-estimate of the elliptic system gives that
    \begin{equation}\label{e-p}
        \begin{aligned}
         \|\nabla^2 u\|_{L^q} \le & C( \|\nabla F\|_{L^q} + \|\nabla w\|_{L^q}  + \|\nabla \mathcal{P}\|_{L^q})\\
        %\le & C (\|H\|_{L^q} + \|\nabla \mathcal{P}\|_{L^q})\\
        %\le& C \bar\rho^{-\alpha}\|-\rho(u_t+u\cdot \nabla u)+2\mu \nabla \rho^\alpha \cdot \mathcal{D}u +\lambda \nabla \rho^{\alpha} \dv u\|_{L^q} \\
        \le& C \left(\bar\rho^{-\alpha}\|\rho u_t\|_{L^q} + \bar \rho^{\gamma-\alpha-1}\|\nabla \rho \|_{L^q}\right)\\
        & +C \left( \bar\rho^{-\alpha+1}  \| u\|_{L^q} \| \nabla u\|_{L^\infty} + \bar\rho^{-1} \|\nabla \rho\|_{L^q} \| \nabla u\|_{L^\infty} \right)\\
        \le& C \left(\bar\rho^{-\alpha}\|\rho u_t\|_{L^q} + \bar \rho^{\gamma-\alpha-1}\|\nabla \rho \|_{L^q}\right)\\
        & + C \left( \bar\rho^{-\alpha+1}  \| u\|_{L^q} \|\nabla u\|_{L^2}^{\theta} \|\nabla^2 u\|_{L^q}^{1-\theta} + \bar\rho^{-1} \|\nabla \rho\|_{L^q} \|\nabla u\|_{L^2}^{\theta} \|\nabla^2 u\|_{L^q}^{1-\theta} \right)\\
        \le & \frac{1}{2} \|\nabla^2 u\|_{L^q} + C \left(\bar\rho^{-\alpha}\|\rho u_t\|_{L^q}  + \bar \rho^{\gamma-\alpha-1}\|\nabla \rho \|_{L^q}\right)\\
         & +  C \left( \bar\rho^{\frac{-\alpha+1}{\theta}}  \| u\|_{L^q}^{\frac{1}{\theta}} \|\nabla u\|_{L^2} + \bar\rho^{-\frac{1}{\theta}} \|\nabla \rho\|_{L^q}^{\frac{1}{\theta}} \|\nabla u\|_{L^2} \right),
        \end{aligned}
    \end{equation}
    where we have used 
    \begin{equation}\label{e3}
    \begin{aligned}
        &\|\nabla F\|_{L^q} + \|\nabla w\|_{L^q} \le C \|H\|_{L^q}\\
        \le& C \bar\rho^{-\alpha}\|-\rho(u_t+u\cdot \nabla u)+2\mu \nabla \rho^\alpha \cdot \mathcal{D}u +\lambda \nabla \rho^{\alpha} \dv u\|_{L^q}\\
        \le& C \left(\bar\rho^{-\alpha}\| \rho u_t\|_{L^q}  + \bar\rho^{-\alpha+1}  \| u\|_{L^q} \| \nabla u\|_{L^\infty} + \bar\rho^{-1} \|\nabla \rho\|_{L^q} \| \nabla u\|_{L^\infty}\right).
    \end{aligned}
    \end{equation}

As a consequence, we have the following estimates which will be used frequently.

\begin{lemma}\label{lem1}
    Let 
    \begin{equation}\label{c1}
        \alpha>1,\quad  \frac{b}{2} \le \min\left\{1, \frac{2(4q-3)}{5q} \alpha - \frac{3(q-2)}{5q} - (\gamma-1), \frac{11q-12}{10q}(\alpha-1) +\frac{3-\gamma}{2}\right\}.
    \end{equation}
    Under the following assumption  %\eqref{a1}, 
    \begin{equation}\label{a11}
    \begin{gathered}
        \mathcal{E}_{\rho,1}(T) \le \frac{\bar\rho}{2}, \quad\mathcal{E}_{\rho,2}(T)  \le  2 N_1 \bar\rho^{b}, \\
		\mathcal{E}_{u,1}(T) \le 2^{\alpha+2} N_2 \bar\rho^{\alpha},
        \quad \mathcal{E}_{u,2}(T) \le 2 N_3 \bar\rho^{2\alpha-1},
    \end{gathered}
    \end{equation}
    with $N_i(i=1,2,3)$ defined in Proposition \ref{pr}, then
    it holds that 
    \begin{gather}\label{key} 
        \sup_{0\leq t\leq T}\|\nabla u\|_{L^6} \le C,\\
        \label{key1}\int_{0}^{T}  \|\nabla u\|_{L^6}^3 dt \le %C \bar \rho^{- ?} \le 
        C,  
    \end{gather}
    where $C=C(a,\lambda,\mu, \alpha,\gamma, \|u_0\|_{H^2}, \|\rho_0-\bar \rho\|_{L^{\gamma}}, \|\nabla \rho_0\|_{L^q})$.
\end{lemma}

\begin{proof}
    Combining \eqref{e1} with \eqref{a11} gives
    \begin{equation}
        \begin{aligned}
            \|\nabla u\|_{L^6} \le& C \left( \bar\rho^{-\alpha+\frac12}\|\sqrt{\rho}u_t\|_{L^2}+ \bar\rho^{-2\alpha+2} \|\nabla u\|_{L^2}^{3}  + \bar\rho^{-\frac{q}{q-3}} \|\nabla \rho \|_{L^q}^{\frac{q}{q-3}} \|\nabla u\|_{L^{2}} \right)\\
        &+ C \bar \rho^{-\alpha+\frac{3(q-2)}{2(4q-3)} + (\gamma-1) \frac{5q}{2(4q-3)} } \|\nabla \rho\|_{L^q}^{\frac{5q}{2(4q-3)}}\\
        \le &C + C \bar\rho^{-2\alpha+2}  + C \bar\rho^{-\frac{q}{q-3}\left(1-\frac{b}{2}\right)}+ C \rho^{-\alpha+\frac{3(q-2)}{2(4q-3)} + (\gamma-1+\frac{b}{2}) \frac{5q}{2(4q-3)} } \le C,
        \end{aligned}
    \end{equation}
    since $\alpha>1, b\le 2,$ and $-\alpha+\frac{3(q-2)}{2(4q-3)} + (\gamma-1+\frac{b}{2}) \frac{5q}{2(4q-3)}\le 0,$ due to \eqref{c1}.
    %if $$\alpha\le 1, b\le 2, -\alpha+\frac{3(q-2)}{2(4q-3)} + (\gamma-1+\frac{b}{2}) \frac{5q}{2(4q-3)} \le 0.$$

    Next, it follows from \eqref{e1} and \eqref{a11} that 
    %\begin{equation}%\label{e1}
    %\begin{aligned}
    %    \int_{0}^{T} \|\nabla u\|_{L^6}^2 dt  
    %    \le & C \left( \bar\rho^{-2\alpha+1}  \int_{0}^{T} \|\sqrt{\rho}u_t\|_{L^2}^2 dt + \bar\rho^{-4\alpha+4}  \int_{0}^{T}  \|\nabla u\|_{L^2}^{2} dt    \right)\\
    %    & + C  \bar\rho^{-\frac{2q}{q-3}} \int_{0}^{T} \|\nabla \rho \|_{L^q}^{\frac{2q}{q-3}} \|\nabla u\|_{L^2}^2 dt \\
    %    & + C \bar \rho^{2\{-\alpha+\frac{3(q-2)}{2(4q-3)} + (\gamma-1) \frac{5q}{2(4q-3)} \}} \int_{0}^{T} \|\nabla \rho\|_{L^q}^{\frac{5q}{(4q-3)}} dt\\
    %    \le & C \left(\bar\rho^{-\alpha+1} + \bar\rho^{-5\alpha+5} + \bar\rho^{(b-2)\frac{q}{q-3} -\alpha+1} + ?\right)\\ 
    %    \le & C \bar\rho^{-\alpha+1}+ ?
    %\end{aligned}
    %\end{equation}
    \begin{equation}%\label{e1}
    \begin{aligned}
        \int_{0}^{T} \|\nabla u\|_{L^6}^{\frac{4(4q-3)}{5q}} dt  
        \le & C \bar\rho^ {\frac{2(4q-3)(-2\alpha+1)}{5q}}  \int_{0}^{T} \|\sqrt{\rho}u_t\|_{L^2}^{\frac{4(4q-3)}{5q}} dt + C  \int_{0}^{T}  \|\nabla u\|_{L^2}^{\frac{4(4q-3)}{5q}} dt \\
        & + C \bar \rho^{\{-\alpha+\frac{3(q-2)}{2(4q-3)} + (\gamma-1) \frac{5q}{2(4q-3)} \}{\frac{4(4q-3)}{5q}}} \int_{0}^{T} \|\nabla \rho\|_{L^q}^2 dt\\
        \le & C \bar\rho^{{\frac{2(4q-3)(-2\alpha+1)}{5q}} } \sup_{0\le t\le T} \|\sqrt{\rho}u_t\|_{L^2}^{\frac{6(q-2)}{5q}} \int_{0}^{T} \|\sqrt{\rho}u_t\|_{L^2}^{2} dt + C  \bar\rho^{-\alpha+1} \\
        & + C \bar \rho^{\{-\alpha+\frac{3(q-2)}{2(4q-3)} + (\gamma-1) \frac{5q}{2(4q-3)} \}{\frac{4(4q-3)}{5q}}+b-\gamma+\alpha} \\
        \le & C \bar\rho^{{\frac{2(4q-3)(-2\alpha+1)}{5q}} + \alpha + \frac{3(q-2)}{5q}(2\alpha-1)}   + C  \bar\rho^{-\alpha+1} \\
        & + C \bar \rho^{\{-\alpha+\frac{3(q-2)}{2(4q-3)} + (\gamma-1) \frac{5q}{2(4q-3)} \}{\frac{4(4q-3)}{5q}}+b-\gamma+\alpha} \\
        \le & C \left(\bar\rho^{-\alpha+1} + \bar \rho^{\{-\alpha+\frac{3(q-2)}{2(4q-3)}  \}{\frac{4(4q-3)}{5q}}+b+\gamma-2+\alpha}\right)\\
        \le & C \left(\bar\rho^{-\alpha+1} + \bar \rho^{\frac{-11q+12}{5q}(\alpha-1) +b+\gamma-3}\right)\le C,
    \end{aligned}
\end{equation}
since $\alpha>1, b<2,$ and $\frac{-11q+12}{5q}(\alpha-1) +b+\gamma-3\le 0,$ due to \eqref{c1}.
    In particular, since $2<\frac{4(4q-3)}{5q}< 3$ for $q\in(3,6)$, \eqref{key} holds. 
\end{proof}

Now, let's concentrate on the estimate of velocities.
\begin{lemma}\label{A3}
    Under the following assumption:
    $$\alpha>\max\left\{1,\frac{\gamma+1}{2}\right\}, \quad b<2,$$
    there exists a positive constant $\Lambda_1=\Lambda_1(a,\lambda,\mu, \alpha,\gamma, \|u_0\|_{H^2}, \|\rho_0-\bar \rho\|_{L^{\gamma}}, \|\nabla \rho_0\|_{L^q})$ such that 
    \begin{equation}
        \mathcal{E}_{u,1}(T)\le  2^{\alpha+1} N_2\bar\rho^{\alpha},
    \end{equation}
    provided $\bar \rho >\Lambda_1$.
    Here, $N_2$ is defined in \eqref{m1}. 

\end{lemma}
\begin{proof}
    Multiplying \eqref{cns}$_2$ by $u_t$, and integrating by parts, we have that 
    \begin{equation}\label{k1}
        \begin{aligned}
            &\frac{d}{dt}\int \left(  \mu\rho^\alpha |\mathcal{D} u|^2+ \frac{\lambda \rho^\alpha}{2}(\dv u)^2\right)dx-\frac{d}{dt}\int  (P(\rho)-P(\bar\rho)) \dv u  dx +  \int\rho |u_t|^2dx \\
            &= -\int  \rho u \cdot \nabla u \cdot u_t dx+ \int \left( \mu(\rho^\alpha)_t |\mathcal{D}u|^2+ \frac{\lambda}{2}(\rho^\alpha)_t |\dv u|^2 \right)dx- \int P_t \dv u dx\\
            &=\sum_{i=1}^{3} I_i.
        \end{aligned}
    \end{equation}
    It follows from H\"older and Sobolev inequalities that 
    \begin{equation}
    \begin{aligned}
        I_1 \le & \frac{1}{2} \|\sqrt{\rho} u_t\|_{L^2}^2 + C \bar \rho \|u\|_{L^6}^2\|\nabla u\|_{L^3}^2\\
        \le & \frac{1}{2} \|\sqrt{\rho} u_t\|_{L^2}^2 + C \bar \rho \|\nabla u\|_{L^2}^3\|\nabla u\|_{L^6} \\
        \le & \frac{1}{2} \|\sqrt{\rho} u_t\|_{L^2}^2 + C \bar \rho \|\nabla u\|_{L^2}^2,
    \end{aligned}
    \end{equation}
    and 
    \begin{equation}
    \begin{aligned}
        I_2&\le C \bar \rho^{\alpha-1} \int  \left|\rho \dv u+ \nabla \rho \cdot u\right||\nabla u|^2 dx\\
        &\le C \bar \rho^\alpha \|\nabla u\|_{L^3}^3 + C \bar \rho^{\alpha-1} \|\nabla \rho \|_{L^q}\|u\|_{L^6} \|\nabla u\|_{L^{\frac{12q}{5q-6}}}^2\\
        &\le C\bar \rho^\alpha \|\nabla u\|_{L^2}^{\frac{3}{2}}\|\nabla u\|_{L^6}^{\frac{3}{2}} + C \bar \rho^{\alpha-1+\frac{b}{2}} \|\nabla u\|_{L^2}^{\frac{5q-6}{2q}} \|\nabla u\|_{L^6}^{\frac{q+6}{2q}}\\ &\le C\bar \rho^\alpha \|\nabla u\|_{L^2}^{\frac{3}{2}}\|\nabla u\|_{L^6}^{\frac{3}{2}} + C \bar \rho^{\alpha} \|\nabla u\|_{L^2}^{\frac{5q-6}{2q}} \|\nabla u\|_{L^6}^{\frac{q+6}{2q}},\\
        %&\le C\bar \rho^\alpha \|\nabla u\|_{L^2}^{\frac{3}{2}}\|\nabla u\|_{L^6}^{\frac{3}{2}} + C \bar \rho^{\alpha-1} \|\nabla \rho \|_{L^q}\|\nabla u\|_{L^2}^{\frac{5q-6}{2q}} \|\nabla u\|_{L^6}^{\frac{q+6}{2q}}\\
        %&\le \frac{M}{2C_1} \bar \rho^\alpha \|\nabla u\|_{L^6}^2 +  C \bar \rho^{\alpha} \|\nabla u\|_{L^2}^{6} + C \bar \rho^{\alpha-\frac{4q}{3(q-2)}} \|\nabla \rho \|_{L^q}^{\frac{4q}{3(q-2)}}\|\nabla u\|_{L^2}^{\frac{2(5q-6)}{3(q-2)}} \\
        %&\le \frac{M}{2C_1} \bar \rho^\alpha \|\nabla u\|_{L^6}^2 +  C \bar \rho^{\alpha} \|\nabla u\|_{L^2}^{2} + C \bar \rho^{\alpha + \frac{2q(b-2)}{3(q-2)}} \|\nabla u\|_{L^2}^2 \\
        %&\le \frac{M}{2C_1} \bar \rho^\alpha \|\nabla u\|_{L^6}^2 +  C \bar \rho^{\alpha} \|\nabla u\|_{L^2}^{2} ,
    \end{aligned}
    \end{equation}
    and that 
    \begin{equation}
        \begin{aligned}
        I_3&\le C \bar \rho^{\gamma-1} \int  \left|\rho \dv u+ \nabla \rho \cdot u\right||\dv u| dx\\
        &\le C \bar \rho^\gamma \|\nabla u\|_{L^2}^2 + C \bar \rho^{\gamma-1} \|\nabla \rho \|_{L^q}\|u\|_{L^\frac{2q}{q-2}} \|\nabla u\|_{L^2}\\
        &\le C\bar \rho^\gamma \|\nabla u\|_{L^2}^{2}+ C \bar \rho^{\gamma-1} \|\nabla \rho \|_{L^q}\|u\|_{L^2}^{\frac{q-3}{q}} \|\nabla u\|_{L^2}^{\frac{3+q}{q}}\\
        &\le C\bar \rho^\gamma \|\nabla u\|_{L^2}^{2}+ C \bar \rho^{\gamma-1+\frac{b}{2}} \|\nabla u\|_{L^2}^2\\
        &\le C\bar \rho^\gamma \|\nabla u\|_{L^2}^{2},
        \end{aligned}
    \end{equation}
    where we have used \eqref{a11} and \eqref{key}.
    Substituting the above estimates into \eqref{k1}, integrating with respect to $t$, and using \eqref{ee} and \eqref{key1} give
    \begin{equation}
        \begin{aligned}
            & \frac{\bar \rho^\alpha}{2^{\alpha+1}}\sup_{0\le t\le T}\int \left(  \mu |\nabla u|^2+ (\mu+\lambda)(\dv u)^2\right)dx + \frac{1}{2} \int_{0}^{T}\|\sqrt{\rho} u_t\|_{L^2}^2 dt \\
            \le
            &\sup_{0\le t\le T}\int \left(  \mu\rho^\alpha |\mathcal{D} u|^2+ \frac{\lambda \rho^\alpha}{2}(\dv u)^2\right)dx + \frac{1}{2} \int_{0}^{T}\|\sqrt{\rho} u_t\|_{L^2}^2 dt \\
            \le&  N_2 2^\alpha \bar \rho^\alpha   + \sup_{0\le  t\le T} \int  (P(\rho)-P(\bar\rho)) \dv u  dx \\
            &+ C \bar \rho^\alpha \int_{0}^{T} \left(\|\nabla u\|_{L^2}^{\frac{3}{2}}\|\nabla u\|_{L^6}^{\frac{3}{2}} +  \|\nabla u\|_{L^2}^{\frac{5q-6}{2q}} \|\nabla u\|_{L^6}^{\frac{q+6}{2q}}\right) dt +  C  \bar \rho^{\gamma} \int_{0}^{T} \|\nabla u\|_{L^2}^{2} dt \\
            \le& N_2 2^\alpha \bar \rho^\alpha + C \bar\rho^{\frac{\gamma}{2}}\|P(\rho)-P(\bar\rho)\|_{L^1}^{\frac{1}{2}} \|\dv u\|_{L^2} +  C (\bar \rho^{\frac{\alpha+1}{2}}+ \bar \rho^{\gamma+1-\alpha}) \\
            \le&  N_2 2^\alpha \bar \rho^\alpha   +  C_1 (\bar\rho^{\frac{\gamma+1}{2}}+ \bar \rho^{\frac{\alpha+1}{2}}+ \bar \rho^{\gamma+1-\alpha})\\
            \le& N_2 2^{\alpha+1} \bar \rho^\alpha,
        \end{aligned}
    \end{equation}
    after defining 
    \begin{equation}\label{m1}
    N_2 \triangleq 2\mu \|\mathcal{D}(u_0) \|_{L^2}^2+\lambda\|\mathrm{div}u_0\|_{L^2}^2, 
    \end{equation}
    and choosing $\Lambda_1$ such that 
    $$\bar \rho \ge \Lambda_1\triangleq \max\left\{1, \left(\frac{3C_1}{2^\alpha N_2}\right)^{\frac{2}{2\alpha-\gamma-1}}, \left(\frac{3C_1}{2^\alpha N_2}\right)^{\frac{2}{\alpha-1}}, \left(\frac{3C_1}{2^\alpha N_2}\right)^{\frac{1}{2\alpha-\gamma-1}}\right\},$$
    as long as $\alpha>\max\left\{1,\frac{\gamma+1}{2}\right\}$.
    The proof is completed.
\end{proof}

Next, we deal with the estimates of $u_t$.

\begin{lemma}\label{A4}
    Assume 
    $$\alpha>\max\left\{1,\frac{\gamma+1}{2}\right\}, \quad b<2,$$
    and 
    $$\frac{5q}{2(4q-3)}(\gamma-1)\le \alpha.$$
    There exists a positive constant $\Lambda_2=\Lambda_2(a,\lambda,\mu, \alpha,\gamma, \|u_0\|_{H^2}, \|\rho_0-\bar \rho\|_{L^{\gamma}}, \|\nabla \rho_0\|_{L^q})$ such that 
    \begin{equation}
        \mathcal{E}_{u,2}(T)\le  2 N_3 \bar\rho^{2\alpha-1},
    \end{equation}
    provided $\bar \rho  >\Lambda_2$. Here, $N_3$ is defined in \eqref{m2}.

\end{lemma}

\begin{proof}
    Differentiating \eqref{cns}$_2$ with repect to $t$, multiplying it by $u_t$ and integrating by parts, we have that 
    \begin{equation}\label{k2}
        \begin{aligned}
            &\frac{1}{2} \frac{d}{dt} \int \rho |u_t|^2 dx + \int  \left(  2 \mu\rho^\alpha |\mathcal{D} u_t|^2+ \lambda \rho^\alpha (\dv u_t)^2\right)dx\\
            =& \int  \dv(\rho u) |u_t|^2 dx -\int  \rho u_t \cdot \nabla u \cdot u_t dx + \int  \dv(\rho u) u \cdot \nabla u \cdot u_t dx  \\
            &+ \int \left( 2 \mu(\rho^\alpha)_t \mathcal{D}u: \mathcal{D}u_t+ \lambda(\rho^\alpha)_t \dv u \dv u_t \right)dx- \int P_t \dv u_t dx\\
            =&\sum_{i=1}^{5} I_i.
        \end{aligned}
    \end{equation}

    It follows from H\"older and Sobolev inequalities, and \eqref{a11}, \eqref{key} that 
    \begin{equation}
        \begin{aligned}
            I_1+I_2\le& C \int  (\bar\rho| \nabla u| +| \nabla \rho \cdot u|) |u_t|^2 dx  \\
            \le &C \left( \bar \rho  \|\nabla u\|_{L^2} +  \|\nabla \rho\|_{L^q} \|u \|_{L^{\frac{2q}{q-2}}}  \right) \| u_t \|_{L^4}^{2} \\
            \le &C \left( \bar \rho^{\frac34}  \|\nabla u\|_{L^2} +  \bar \rho^{-\frac14} \|\nabla \rho\|_{L^q} \|u \|_{L^{2}}^{\frac{q-3}{q}} \|\nabla u \|_{L^{2}}^{\frac{3}{q}} \right) \|\sqrt{\rho} u_t \|_{L^2}^{\frac12} \|\nabla u_t \|_{L^2}^{\frac32} \\
            \le & \frac{\mu}{2^{\alpha+3}} \bar \rho^{\alpha} \|\nabla u_t\|_{L^2}^2 + C\left(\bar \rho^{3-3\alpha} \|\nabla u\|_{L^2}^4  +   \bar \rho^{-3\alpha-1} \|\nabla \rho\|_{L^q}^4  \|\nabla u \|_{L^{2}}^{\frac{12}{q}}  \right) \|\sqrt{\rho} u_t \|_{L^2}^2 \\
            \le & \frac{\mu}{2^{\alpha+3}} \bar \rho^{\alpha} \|\nabla u_t\|_{L^2}^2 + C\left(\bar \rho^{3-3\alpha} \|\nabla u\|_{L^2}^2  +   \bar \rho^{-3\alpha-1+2b}   \|\nabla u \|_{L^{2}}^{2}  \right) \|\sqrt{\rho} u_t \|_{L^2}^2\\
            \le & \frac{\mu}{2^{\alpha+3}} \bar \rho^{\alpha} \|\nabla u_t\|_{L^2}^2 + C \bar \rho^{2-\alpha} \|\nabla u\|_{L^2}^2,   %\|\sqrt{\rho} u_t \|_{L^2}^2,
        \end{aligned}
    \end{equation}
    also we can deduce 
    \begin{equation}
        \begin{aligned}
            I_3\le & \int  |\rho \dv u+ \nabla \rho \cdot u| |u| | \nabla u| |u_t| dx\\
            \le & C \left(\bar \rho \|u\|_{L^6} \|\nabla u\|_{L^3}^2  \| u_t\|_{L^6} + \|\nabla \rho\|_{L^q}  \|u\|_{L^6}^2  \|\nabla u\|_{L^{\frac{2q}{q-2}}} \|u_t\|_{L^6} \right)\\
            \le & \frac{\mu}{2^{\alpha+3}} \bar \rho^{\alpha} \|\nabla u_t\|_{L^2}^2 + C \bar \rho^{2-\alpha} \|\nabla u\|_{L^2}^4 \|\nabla u\|_{L^6}^2 + C \bar \rho^{ - \alpha} \|\nabla \rho \|_{L^q}^2  \|\nabla u\|_{L^2}^2\\
            \le & \frac{\mu}{2^{\alpha+3}} \bar \rho^{\alpha} \|\nabla u_t\|_{L^2}^2 + C \bar \rho^{2-\alpha} \|\nabla u\|_{L^2}^2  + C \bar \rho^{ - \alpha+b}  \|\nabla u\|_{L^2}^2\\
            \le & \frac{\mu}{2^{\alpha+3}} \bar \rho^{\alpha} \|\nabla u_t\|_{L^2}^2 + C \bar \rho^{2-\alpha} \|\nabla u\|_{L^2}^2  .
        \end{aligned}
    \end{equation}
    Similarly, we can estimate $I_4$ as follows:
    \begin{equation}
        \begin{aligned}
            I_4\le & C  \bar \rho^{\alpha-1}\int |\rho \dv u+ \nabla \rho \cdot u| |\nabla  u||\nabla  u_t| dx\\
            \le & C \left(\bar \rho^{\alpha} \|\nabla u\|_{L^4}^2 \|\nabla  u_t\|_{L^2} + \bar \rho^{\alpha-1}  \|\nabla \rho\|_{L^q}  \|u\|_{L^\frac{3q}{q-3}}  \|\nabla u\|_{L^6} \|\nabla u_t\|_{L^2} \right)\\
            \le & \frac{\mu}{2^{\alpha+3}} \bar \rho^{\alpha} \|\nabla u_t\|_{L^2}^2 + C \bar \rho^{\alpha} \|\nabla u\|_{L^2} \|\nabla u\|_{L^6}^3 + C \bar \rho^{\alpha-2} \|\nabla \rho \|_{L^q}^2  \|\nabla u\|_{L^2}^{\frac{3(q-2)}{q}} \|\nabla u\|_{L^6}^{2+\frac{6-q}{q}}\\
            \le & \frac{\mu}{2^{\alpha+3}} \bar \rho^{\alpha} \|\nabla u_t\|_{L^2}^2 + C \bar \rho^{\alpha} \|\nabla u\|_{L^2} \|\nabla u\|_{L^6}^3 + C \bar \rho^{\alpha-2+b}  \|\nabla u\|_{L^2}^{\frac{3(q-2)}{q}} \|\nabla u\|_{L^6}^{2+\frac{6-q}{q}}\\
            \le & \frac{\mu}{2^{\alpha+3}} \bar \rho^{\alpha} \|\nabla u_t\|_{L^2}^2 + C \bar \rho^{\alpha} \|\nabla u\|_{L^2} \|\nabla u\|_{L^6}^3 + C \bar \rho^{\alpha}  \|\nabla u\|_{L^2}^{\frac{3(q-2)}{q}} \|\nabla u\|_{L^6}^{2+\frac{6-q}{q}},
        \end{aligned}
    \end{equation}
    and $I_5$ is bounded by
    \begin{equation}
        \begin{aligned}
            I_5\le& C  \bar \rho^{\gamma-1}\int |\rho \dv u+ \nabla \rho \cdot u| |\dv u_t| dx \\
            \le & C \left(\bar \rho^{\gamma} \|\dv u\|_{L^2} \|\dv u_t\|_{L^2} + \bar \rho^{\gamma-1}  \|\nabla \rho\|_{L^q}  \|u\|_{L^{\frac{2q}{q-2}}} \|\dv u_t\|_{L^2} \right)\\
            \le & \frac{\mu}{2^{\alpha+3}} \bar \rho^{\alpha} \|\nabla u_t\|_{L^2}^2 + C \bar \rho^{2\gamma-\alpha} \|\nabla u\|_{L^2}^2 + C \bar \rho^{2\gamma-2-\alpha} \|\nabla \rho \|_{L^q}^2 \|u\|_{L^2}^{\frac{2(q-3)}{q}} \|\nabla u\|_{L^2}^{\frac{6}{q}}\\
            \le & \frac{\mu}{2^{\alpha+3}} \bar \rho^{\alpha} \|\nabla u_t\|_{L^2}^2 + C \bar \rho^{2\gamma-\alpha} \|\nabla u\|_{L^2}^2 + C \bar \rho^{2\gamma-2-\alpha}  \|\nabla \rho\|_{L^q}^{2}.
            %\\ \le & \frac{\mu}{2^{\alpha+3}} \bar \rho^{\alpha} \|\nabla u_t\|_{L^2}^2 + C \bar \rho^{2\gamma-\alpha} \|\nabla u\|_{L^2}^2 . 
        \end{aligned}
    \end{equation}

    Integrating \eqref{k2} and using the above estimates and \eqref{key1}, we have 
    \begin{equation}\label{k3}
        \begin{aligned}
            &\frac{1}{2} \sup_{0\le t\le T}  \|\sqrt{\rho} u_t \|_{L^2}^2 +   \mu \frac{\bar \rho^\alpha}{2^\alpha} \int_{0}^{T} \|\nabla u_t\|^2_{L^2} dt+ (\mu+\lambda) \frac{\bar \rho^\alpha}{2^\alpha} \int_{0}^{T} \|\dv u_t\|_{L^2}^2 dt\\
            \le &  \frac{1}{2} \sup_{0\le t\le T}  \|\sqrt{\rho} u_t \|_{L^2}^2 + \int_{0}^{T} \int  \left(  2 \mu\rho^\alpha |\mathcal{D} u_t|^2+ \lambda \rho^\alpha (\dv u_t)^2\right)dx dt \\
            \le & \frac{1}{2} \|\sqrt{\rho_0} u_{0t} \|_{L^2}^2+ \frac{\mu}{2^{\alpha+1}} \bar \rho^\alpha \int_{0}^{T} \|\nabla u_t\|^2_{L^2} dt + C \bar \rho^{2\gamma-2-\alpha} \int_{0}^{T}  \|\nabla \rho\|_{L^q}^{2} dt \\
            &+   C \left( \bar \rho^{2-\alpha} + \bar \rho^{2\gamma-\alpha} \right) \int_{0}^{T} \|\nabla u\|_{L^2}^2 dt \\
            &+ C \bar \rho^{\alpha}  \bigg( \int_{0}^{T} \|\nabla u\|_{L^2} \|\nabla u\|_{L^6}^3 dt + \int_{0}^{T} \|\nabla u\|_{L^2}^{\frac{3(q-2)}{q}} \|\nabla u\|_{L^6}^{2+\frac{6-q}{q}} dt \bigg)\\
            \le &  \frac{1}{2} \|\sqrt{\rho_0} u_{0t} \|_{L^2}^2+ \frac{\mu}{2^{\alpha+1}} \bar \rho^\alpha \int_{0}^{T} \|\nabla u_t\|^2_{L^2} dt +  C \left(  \bar \rho^{\alpha} + \bar \rho^{2\gamma-2\alpha +1}  + \bar \rho^{\gamma-2 +b} \right).  
        \end{aligned}
    \end{equation}

    Define 
    \begin{equation*}
        \rho_0^{1/2} u_{t0} \triangleq \rho_0^{-1/2}\left(-\rho_0 u_0\cdot \nabla u_0 - \nabla P(\rho_0) +2\mu \dv( \rho_0^{\alpha}\mathcal{D} u_0)+\lambda\nabla( \rho_0^{\alpha} \dv u_0)\right),
    \end{equation*}
    then it follows from \eqref{ia} that 
    \begin{equation}\label{ki}
    \begin{aligned}
        \|\sqrt{\rho_0} u_{t0}\|_{L^2} \le& C 
        \left(\bar \rho^{\frac12} \|\nabla u_0\|_{L^2}^{\frac{3}{2}} \|\nabla u_0\|_{L^6}^{\frac{1}{2}} + \bar \rho^{-\frac{1}{2}}\|\nabla P(\rho_0)\|_{L^2}\right) \\
        &+ C \left(\bar \rho^{\alpha-\frac{1}{2}}\|\nabla^2 u_0\|_{L^2} + \bar \rho^{\alpha-\frac{3}{2}} \|\nabla \rho_0\|_{L^q} \|\nabla u_0\|_{L^{\frac{2q}{q-2}}} \right)\\
        \le& C  \left(\bar \rho^{-\frac{1}{2}} \|P(\rho_0)-{P}(\bar \rho)\|_{L^1}^{\frac{3(q-2)}{2(4q-3)}} \|\nabla P(\rho_0)\|_{L^q}^{\frac{5q}{2(4q-3)}}+ \bar \rho^{\alpha-\frac{1}{2}}\right)\\
        \le & C\left(\bar \rho^{-\frac{1}{2}+ \frac{5q(\gamma-1)}{2(4q-3)}}+ \bar \rho^{\alpha-\frac{1}{2}}\right)\\
        \le & C_2 \bar \rho^{\alpha-\frac{1}{2}},
    \end{aligned}
    \end{equation}
    due to 
    $\frac{5q}{2(4q-3)}(\gamma-1)\le \alpha.$

    By combining \eqref{k3} with \eqref{ki}, we obtain
    \begin{equation}
        \begin{aligned}
            &\sup_{0\le t\le T}  \|\sqrt{\rho} u_t \|_{L^2}^2 +   \mu \frac{\bar \rho^\alpha}{2^{\alpha+1}} \int_{0}^{T} \|\nabla u_t\|^2_{L^2} dt\\
            \le &  N_3 \bar \rho^{2\alpha-1}  +  C_3 \left( \bar \rho^{\alpha} + \bar \rho^{2\gamma-2\alpha +1} + \bar \rho^{\gamma} \right) \\
            \le & 2 N_3 \bar \rho^{2\alpha-1},  
        \end{aligned}
    \end{equation}
    after defining
    \begin{equation}\label{m2}
       N_3=N_3(a,\mu,\lambda, \alpha, \gamma, \|\rho_0-\bar \rho\|_{L^\gamma}, \|\nabla \rho_0\|_{L^q}, \|\nabla u_0\|_{H^1}) \triangleq C_2^2, 
    \end{equation}
    and choosing 
    $$\bar \rho \ge \Lambda_2 \triangleq \max\left\{\Lambda_1, \left(\frac{3C_3}{ N_3}\right)^{\frac{1}{\alpha-1}}, \left(\frac{3C_3}{ N_3}\right)^{\frac{1}{2(2\alpha-\gamma-1)}}, \left(\frac{3C_3}{ N_3}\right)^{\frac{1}{2\alpha-\gamma-1}} \right\},$$
    as long as $\alpha>\max\left\{1,\frac{\gamma+1}{2}\right\}$.
\end{proof}

Now, the next lemma is concentrated on the estimates of density.

\begin{lemma}\label{A2}
    Assume that 
    \begin{equation}\label{c2}
        \begin{cases}
            3-\gamma\le b<2, &1<\gamma<3,\\
            b=0, &\gamma\ge 3.
        \end{cases}
    \end{equation}
    Then it holds that 
    \begin{equation}
        \mathcal{E}_{\rho,2}(T)\le   N_1 \bar\rho^{b},
    \end{equation}
    provided $\bar \rho>\Lambda_2$. Here, $N_1$ is defined in \eqref{m3}.
\end{lemma}

\begin{proof}
    It follows from \eqref{cns}$_1$ and \eqref{flux} that  
    \begin{equation}
        \nabla \rho_t + u \cdot \nabla^2 \rho + \nabla u \cdot \nabla \rho + \nabla \rho \dv u+ \frac{1}{2\mu+\lambda} \rho \left(\nabla F + \nabla \mathcal{P} \right)=0.
    \end{equation}
    Multiplying the above equation by $|\nabla \rho|^{q-2}\nabla \rho$ and then integrating by parts, we have 
    \begin{equation}
        \begin{aligned}
            &\frac{1}{p}\frac{d}{dt} \|\nabla \rho\|_{L^q}^q + \frac{a \gamma }{2\mu+\lambda}\int  \rho^{\gamma-\alpha}|\nabla \rho|^qdx \\
            =& \frac{1-q}{q}\int \dv u |\nabla \rho |^qdx - \int \nabla \rho \cdot \nabla u \cdot \nabla \rho |\nabla \rho|^{q-2} dx - \frac{ 1 }{2\mu+\lambda} \int \rho \nabla F \cdot \nabla \rho |\nabla \rho|^{p-2} dx \\
            \le & C \|\nabla u\|_{L^\infty} \|\nabla \rho \|_{L^q}^q + C \bar \rho \|\nabla F\|_{L^q} \|\nabla \rho \|_{L^q}^{q-1}.
        \end{aligned}
    \end{equation}
    Then, we have
    \begin{equation}
        \frac{d}{dt} \|\nabla \rho\|_{L^q}^2 + \bar \rho^{\gamma-\alpha}   \|\nabla \rho\|_{L^q}^2  \le C \|\nabla u\|_{L^\infty} \|\nabla \rho \|_{L^q}^2 + C \bar \rho \|\nabla F\|_{L^q} \|\nabla \rho \|_{L^q},
    \end{equation}
    which together with \eqref{e3} gives that 
    \begin{equation}
    \begin{aligned}
        &\frac{d}{dt} \|\nabla \rho\|_{L^q}^2 + \bar \rho^{\gamma-\alpha}   \|\nabla \rho\|_{L^q}^2\\
        \le& C \|\nabla u\|_{L^\infty} \|\nabla \rho \|_{L^q}^2 + C \bar\rho^{-\alpha+1} \|\rho u_t\|_{L^q} \|\nabla \rho \|_{L^q} + C \bar\rho^{-\alpha+2}  \| u\|_{L^q} \| \nabla u\|_{L^\infty} \|\nabla \rho \|_{L^q}\\
        \le & \frac12 \bar \rho^{\gamma-\alpha}   \|\nabla \rho\|_{L^q}^2 + C \bar \rho^{\alpha-\gamma} \|\nabla u\|_{L^\infty}^2 \|\nabla \rho \|_{L^q}^2 + C \bar\rho^{2-\alpha-\gamma} \|\rho u_t\|_{L^q}^2 + C \bar\rho^{4-\alpha-\gamma}   \| \nabla u\|_{L^\infty}^2.
    \end{aligned}
    \end{equation}

    Next, it follows from \eqref{a11} and \eqref{e-p} that 
    \begin{equation}\label{ep1}
        \begin{aligned}
            \int_{0}^{T}  \|\rho u_t\|_{L^q}^2 dt \le & \int_{0}^{T}  \bar\rho^{\frac{5q-6}{2q}}\|\sqrt{\rho}u_t\|_{L^2}^{\frac{6-q}{q}} \|\nabla u_t\|_{L^2}^{\frac{3(q-2)}{q}} dt\\
            \le & C \bar\rho^{\frac{5q-6}{2q}} \left(\int_{0}^{T} \|\sqrt{\rho}u_t\|_{L^2}^2 dt\right)^{\frac{6-q}{2q}} \left( \int_{0}^{T}\|\nabla u_t\|_{L^2}^2 dt \right)^{\frac{3(q-2)}{2q}}\\
            \le & C \bar\rho^{\frac{5q-6}{2q}+ \alpha \frac{6-q}{2q} + (\alpha-1)  \frac{3(q-2)}{2q}} = C \bar\rho^{1+\alpha},
        \end{aligned}
    \end{equation}
    and that 
    \begin{equation}\label{ep2}
        \begin{aligned}
            \int_{0}^{T}  \|\nabla u\|_{L^\infty}^2 dt 
            \le & C \int_{0}^{T} \|\nabla u\|_{L^2}^{\theta} \|\nabla^2 u\|_{L^q}^{1-\theta} dt\\
            %\le& C \int_{0}^{T} \|\nabla u\|_{L^2}^{2 \theta} \left( \bar\rho^{-\alpha}\|\rho u_t\|_{L^q} + \bar\rho^{\frac{-\alpha+1}{\theta}}  \| u\|_{L^q}^{\frac{1}{\theta}} \|\nabla u\|_{L^2} + \bar\rho^{-\frac{1}{\theta}} \|\nabla \rho\|_{L^q}^{\frac{1}{\theta}} \|\nabla u\|_{L^2}  + \bar \rho^{\gamma-\alpha-1}\|\nabla \rho \|_{L^q}\right)^{2(1-\theta)} dt \\
            \le & C \int_{0}^{T}  \biggl( \|\nabla u\|_{L^2}^{2 } 
            + \bar\rho^{-2\alpha}\|\rho u_t\|_{L^q}^2 + \bar\rho^{\frac{-2\alpha+2}{\theta}}  \| u\|_{L^q}^{\frac{2}{\theta}} \|\nabla u\|_{L^2}^2 \\
            & +\bar\rho^{-\frac{2}{\theta}} \|\nabla \rho\|_{L^q}^{\frac{2}{\theta}} \|\nabla u\|_{L^2}^2  
            + \bar \rho^{2\gamma-2\alpha-2}\|\nabla \rho \|_{L^q}^2\biggr) dt\\
            \le & C \int_{0}^{T}  \left( \|\nabla u\|_{L^2}^{2 } 
            + \bar\rho^{-2\alpha}\|\rho u_t\|_{L^q}^2  + \bar \rho^{2\gamma-2\alpha-2}\|\nabla \rho \|_{L^q}^2 \right) dt \\
            \le & C  \bar\rho^{1-\alpha} 
            + C\rho^{\gamma-\alpha-2 + b}\\
            \le & C \rho^{\gamma-\alpha},
        \end{aligned}
    \end{equation}
    with $\theta$ defined in \eqref{theta}.

    Finally, applying the Gronwall inequality to \eqref{rho1} and using \eqref{ep1}, \eqref{ep2}, we have 
    \begin{equation}
    \begin{aligned}
        &\sup_{0\le t\le T} \|\nabla \rho\|_{L^q}^2 + \frac{1}{2} \bar \rho^{\gamma-\alpha}   \int_{0}^{T} \|\nabla \rho\|_{L^q}^2 dt \\
        \le& \exp{\left\{\int_{0}^{T} C \bar \rho^{\alpha-\gamma} \|\nabla u\|_{L^\infty}^2 dt \right\}}\\
        &\cdot\left(\|\nabla \rho_0\|_{L^q}^2 + C \bar\rho^{2-\alpha-\gamma} \int_{0}^{T}\|\rho u_t\|_{L^q}^2 dt + C \bar\rho^{4-\alpha-\gamma}  \int_{0}^{T} \| \nabla u\|_{L^\infty}^2 dt\right)\\
        \le & C \left(\|\nabla \rho_0\|_{L^q}^2 + C \bar\rho^{3-\gamma} + C \bar\rho^{4-2\alpha}  \right)\\
        \le & C_4 \left(\|\nabla \rho_0\|_{L^q}^2 +  \bar\rho^{3-\gamma} \right)\\
        \le & N_1 \bar \rho^{b},
    \end{aligned}
    \end{equation}
    where $b$ satisfies \eqref{c2} and
    \begin{equation}\label{m3}
        N_1 \triangleq C_4 \left(\|\nabla \rho_0\|_{L^q}^2 + 1\right).
    \end{equation}
    % and choosing 
    % \begin{equation}
    %     \begin{cases}
    %         b \ge 3-\gamma, \quad &\gamma<3,\\
    %         b =0, & \gamma\ge 3,
    %     \end{cases}
    % \end{equation}
    
    %$$\bar \rho \ge \Lambd\mathcal{E}_{u,1}\triangleq \max\{\Lambd\mathcal{E}_{\rho,2}, \}. $$
\end{proof}

\begin{proof}[The proof of the Proposition \ref{pr}]
    
    It only suffices to prove that 
    \begin{equation}
    \begin{aligned}
        \mathcal{E}_{\rho,1}(T)=\sup_{t\in[0,T] }\|\rho-\bar\rho \|_{L^\infty}\le &C \sup_{t\in[0,T] }\|\rho-\bar\rho \|_{L^\gamma}^{\frac{(q-3)\gamma }{3q+(q-3)\gamma}} \|\nabla \rho \|_{L^q}^{\frac{3q }{3q+(q-3)\gamma}}\\
        \le& C_5 \bar \rho^{\frac{(q-3)\gamma }{3q+(q-3)\gamma}+ \frac{b}{2} \frac{3q }{3q+(q-3)\gamma} } \le \frac{\bar \rho}{4}, 
    \end{aligned}
    \end{equation}
    provided that 
    \begin{equation}\label{lambda}
        \bar \rho \ge \Lambda_0 \triangleq \max \left\{\Lambda_2, (4C_5)^{\frac{2(3q+(q-3)\gamma)}{3q(2-b)}} \right\},
    \end{equation}
    and 
    $$b<2.$$

    Finally, choose $\beta$ as in \eqref{b}, then if 
    $$\alpha>1, \gamma>1, 2\alpha>\gamma+1,$$
    and 
    $$\frac{5q}{2(4q-3)}(\gamma-1)\le \alpha,$$
    $\beta$ satisfies \eqref{c1} and \eqref{c2} automatically. Hence,
    \eqref{a2} follows from Lemmas  \ref{A3}, \ref{A4}, \ref{A2}, and the proof is finished. 
\end{proof}

\section{Proof of the Theorem \ref{th1}}\label{s4}

By Lemma \ref{local}, there exists a positive time $T_*$ such that the Cauchy problem \eqref{cns}-\eqref{7} has a strong solution $(\rho,u)$ in $\mathbb{R}^3\times (0, T_*]$. So, to prove Theorem \ref{th1}, it suffices to prove that the local solution can be extended to be a global one. To this end, we assume from now on that $\rho \ge \Lambda_0$ holds with $\Lambda_0$ being defined in Proposition \ref{pr}. 

Set 
\begin{equation}\label{4.1}
    T^*=\sup\{T|(\rho,u) \text{ is a strong solution on } [0,T]\}. 
\end{equation}

We claim that 
\begin{equation}\label{4.2}
    T^*=\infty.
\end{equation}
In fact, if $T^*<\infty$. By virtue of Proposition \ref{pr}, it holds that $(\rho,u)|_{t=T^*}$ satisfies \eqref{ia}. Thus, Lemma \ref{local} implies that there exists some $T^{**}>T^*$ such that $(\rho,u)$ can be extended to be a strong solution to \eqref{cns}-\eqref{7} in $\mathbb{R}^3\times (0, T_{**}]$, which contradicts \eqref{4.1}. Hence, \eqref{4.2} holds.

\section*{Conflict-of-interest statement}
All authors declare that they have no conflicts of interest.

\section*{Acknowledgments}
X.-D. Huang is partially supported by NNSFC Grant Nos. 11971464, 11688101 and CAS
Project for Young Scientists in Basic Research, Grant No. YSBR-031, National Key R$\&$D Program of China, Grant No. 2021YFA1000800.
J.-X. Li was supported in part
by Zheng Ge Ru Foundation, Hong Kong RGC Earmarked Research Grants CUHK-14301421, CUHK-14300819,
CUHK-14302819, CUHK-14300917, the key project of NSFC (Grant No. 12131010)  the Shun Hing Education and Charity Fund.
Part of this work was done when R. Zhang was visiting the Institute of Mathematical Sciences at the Chinese University of Hong Kong. They would like to thank the institute for its hospitality.

\bibliographystyle{siam}

\end{document}